%% file: main.tex
\documentclass[preprint]{elsarticle}
\usepackage{graphicx} 
\usepackage[scale=0.75]{geometry}

\journal{}
\usepackage{stmaryrd}
\usepackage{appendix}
\input{notations.tex}

\begin{document}

\begin{frontmatter}



\title{On factorization of rank-one auto-correlation matrix polynomials}


\author[cranaddress]{Konstantin Usevich}
\ead{konstantin.usevich@cnrs.fr}
\author[cranaddress]{Julien Flamant}
\ead{julien.flamant@cnrs.fr}
\author[iecladdress]{Marianne Clausel}
\ead{marianne.clausel@univ-lorraine.fr}
\author[cranaddress]{David Brie}
\ead{david.brie@univ-lorraine.fr}
\address[cranaddress]{Université de Lorraine, CNRS, CRAN, F-54000 Nancy France}
\address[iecladdress]{Université de Lorraine, CNRS, Institut Elie Cartan de Lorraine, F-54000 Nancy France}

\begin{abstract}
This article characterizes the rank-one factorization of auto-correlation matrix polynomials. 
We establish a sufficient and necessary uniqueness condition for uniqueness of the factorization based on the greatest common divisor (GCD) of multiple polynomials. 
In the unique case, we show that the factorization can be carried out explicitly using GCDs. 
In the non-unique case, the number of non-trivially different factorizations is given and all solutions are enumerated.
\end{abstract}


\begin{keyword}
matrix polynomial\sep auto-correlation \sep rank-one factorization \sep uniqueness \sep greatest common divisors
\MSC{15A21,
15A23,
42A85,
47A68,
47N70}
\end{keyword}
\end{frontmatter}
\section{Introduction}
Let $\boldsymbol{\Gamma}(z) = [\Gamma_{ij}(z)]_{i,j=1}^{K,K}$ be an $K \times K$ matrix polynomial of degree at most $2(N-1)$, with complex coefficients.
The goal of this paper is to characterize the matrix polynomials that admit the following rank-one factorization:
\begin{equation}\label{prob:PAF}
\boldsymbol{\Gamma}(z) = \begin{bmatrix}
        \Gamma_{11}(z) & \cdots &\Gamma_{1K}(z)\\
        \vdots & & \vdots\\
        \Gamma_{K1}(z) & \cdots &\Gamma_{KK}(z)\\
    \end{bmatrix} =
\begin{bmatrix}
X_1(z) \\
\vdots \\
X_K(z) 
\end{bmatrix}
\begin{bmatrix}
\widetilde{X}_1(z) &
\cdots &
\widetilde{X}_K(z) 
\end{bmatrix},
\end{equation}
where $X_k(z), k \in \{1,\ldots, K\}$ are  polynomials with degree at most $N-1$;
in \eqref{prob:PAF}, the notation  $\widetilde{Y}(z)$ denotes the complex conjugate reversal
\begin{equation}\label{eq:conjReverse}
\widetilde{Y}(z) := z^{N-1} \conj{{Y}(\conj{z}^{-1})}  =  \sum_{n=0}^{N-1} \conj{y[N-1-n]}z^n
\end{equation}
of  a polynomial ${Y}(z) = \sum_{n=0}^{N-1} \conj{y[n]}z^n$ of degree at most $N-1$.
In this paper we address the following questions:
\begin{itemize}
\item when is the factorization \eqref{prob:PAF} unique?
\item if it is not unique, how to find all possible factorizations \eqref{prob:PAF}?
\end{itemize}
The polynomial factorization problem arises in several applications in signal processing, such as phase retrieval problems \cite{flamant2023polarimetric} and blind  system identification \cite{jaganathan_reconstruction_2019}. In such applications, one is interested to reconstruct a number of signals (vectors)
$\boldsymbol{x}_k = \begin{bmatrix} x_k[0] & \cdots & x_k[N-1] \end{bmatrix}^{\transp} \in \bbC^{N}, k \in \{1,\ldots,K\}$, given a set of their correlations sequences $\left(\gamma_{ij}[n]\right)_{i,j =1, n=-N+1}^{K,K,N-1}$ defined as follows:
\begin{equation}\label{eq:autocor}
\gamma_{ij}[n] := \sum_{m=0}^{N-1-n} x_i[m+n]\conj{x_j[m]}.
\end{equation} 
Such correlation functions encode the $K\times K$ auto-correlation matrix sequence of the $K$-dimensional vector signal $([x_1[n]\:\ldots \: x_K[n]]^\transp)_{n=0}^{N-1})$.
Moreover it can be shown (see Appendix \ref{app:autocor}) that the elements  in \eqref{eq:autocor} corresponds exactly to the coefficients of the polynomial $\Gamma_{ij}(z)$ factorized as \eqref{prob:PAF}.
Thus the problem of recovery of the vectors $\boldsymbol{x}_k$ from correlations sequences is equivalent to  the problem of factorizing\footnote{This is the reason why in \cite{flamant2023polarimetric} we refer to \eqref{prob:PAF} as the polynomial auto-correlation factorization (PAF) problem.} a given matrix polynomial as \eqref{prob:PAF}.

In this paper we provide a complete characterization of all possible factorizations rank-one matrix polynomial \eqref{prob:PAF}; in fact, these factorizations are entirely characterized by the greatest common divisor of all the matrix elements $\Gamma_{ij}(z)$, denoted as $\gcd\lbrace\Gamma_{ij}\rbrace_{i,j=1}^{K,K}$.
In particular, we prove the following theorem.

\begin{theorem}\label{eq:thm_main}
The factorization \eqref{prob:PAF} is unique up to global scaling if and only if the greatest common divisor  $H(z) =  \gcd\lbrace\Gamma_{ij}\rbrace_{i,j=1}^{K,K}$ is $1$ or it roots lie only on  the complex unit circle $\bbT := \lbrace z \in \bbC \mid \vert z \vert =1\rbrace$.
\end{theorem}
By uniqueness up to global scaling in \eqref{eq:thm_main} we mean that any alternative factorization $\Gamma_{ij} (z) = Y_i(z) \widetilde{Y}_j (z)$ satisfies $Y_i(z) = c X_i(z)$ with some $c \in \bbT$.
Moreover, in the non-unique case, we provide all possible factorizations modulo the global scaling, which again depend on roots of the polynomial $H(z)$ and their multiplicities.

\emph{Related work.}
Factorizations of matrix polynomials and matrix functions are a classic topic in linear algebra and operator theory \cite{gohberg2009matrix}.
In fact, it can be shown that the matrix polynomials factorized as \eqref{prob:PAF} have the so-called $*$-palindromic structure \cite{mackey2011smith}.
Several previous works have addressed the spectral properties or the Smith normal form of palindromic matrix polynomials, but, up to the authors knowledge, none of them discussed in detail uniqueness and characterization of solutions, which is a very important question in applications mentioned above.
The factorization \eqref{prob:PAF} also resembles the problem of spectral factorization of matrix functions \cite{bini2003effective}, however,  unlike the latter problem, in the factorization \eqref{prob:PAF} there is no restriction on location of roots of the polynomials inside the unit disk.
Very few papers on low-rank factorization treat the low-rank case \cite{ephremidze2015rank}.
It is shown that \cite[Theorem 1]{ephremidze2015rank} that any rank-one matrix polynomial that is positive-semidefinite on the complex unit circle $\bbT$ admits a factorization \eqref{prob:PAF} that is unique if one imposes additional constraints (minimum phase requirement).
The question of characterizing the set  of non-minimal-phase spectral factorizations was  only analyzed \cite{ephremidze2017non} for the full rank case; also, only the generic case (determinant with only simple roots) is treated.

The uniqueness of the factorization \eqref{prob:PAF} is also closely related to uniqueness of the solution in (algebraic) phase retrieval problems (see \cite{flamant2023polarimetric}).
For the latter problem, the case $K=1$ (corresponding to $1\times 1$ matrices) was fully characterized in  \cite{beinert2015ambiguities}. 
In the case $K>1$, only partial results are available in \cite{jaganathan_reconstruction_2019}, while our paper provides a complete characterization in the language of matrix polynomial factorizations.

\emph{Caveat.}
In this paper, we employ the formalism of univariate polynomials with roots at infinity, that is crucial for the algebraic theory  of Hankel matrices \cite[\S I.0]{heinig1984algebraic} and matrix methods for approximate greatest common divisor computations \cite{usevich_variable_2017}.
This formalism simplifies the proofs and allows for a transparent and complete characterization for the polynomial factorization problem.
However,  the notion of greatest common divisor is slightly different from the standard one, as it takes into account the possible roots at infinity.

\emph{Organization of the paper.}
The main notation and main facts regarding polynomials with roots at infinity is surveyed in \Cref{sec:background}.
In \Cref{sub:uniquenessGeneral}, we provide the main result on uniqueness of the factorization, which is a slight generalization of \Cref{eq:thm_main}.
Finally, in \Cref{sec:nonuniqueness} we discuss the complete description of the set of solutions.

\section{Background: polynomials with roots at infinity}\label{sec:background}
The goal of this subsection is to introduce the formalism of the polynomials with the roots at infinity used later in the paper.
As mentioned in \cite{usevich_variable_2017}, such spaces of polynomials can be  identified with the space of homogeneous bivariate polynomials with roots in a projective space (an approach commonly used in algebraic geometry \cite[Ch. 8]{cox_ideals_1997}).
In this paper, however, for simplicity we prefer to work with univariate polynomials in $\bbC_{\le D}[z]$ instead.

\subsection{Vector spaces of polynomials of bounded degree}
Let $\bbC$ denote the complex field, $\bbT = \{  z \in \bbC \, | \, |z| = 1 \}$  denote the unit circle, and let $\bbC_{\le D}[z]$ denote the space of univariate polynomials with complex coefficients of degree at most $D$.
Any polynomial $A \in \bbC_{\le D}[z]$ is in one-to-one correspondence with its vector of coefficients:
\begin{equation}\label{eq:poly}
A(z) = \sum\limits_{n=0}^{D} a[n] z^n \quad  \leftrightarrow \quad \bfa =
  \begin{bmatrix}
    a[0] & a[1] & \cdots & a[D]
  \end{bmatrix}^{\transp};
\end{equation}
thus $\bbC_{\le D}[z]$ is a $(D+1)$-dimensional  vector space that is isomorphic to  $\bbC^{D+1}$.
In what follows, we are going to use the notation $\bfa$, $A(z)$, $a[n]$ for the vectors, polynomials and coefficients of polynomials/vectors (with some abuse of meaning for the brackets). 

The conjugate reversal $\widetilde{A}(z)$ of the polynomial in \eqref{eq:poly} is given by \eqref{eq:conjReverse} and corresponds to the complex-conjugated and reversed vector of coefficients 
\[
\widetilde{\bfa}  = \begin{bmatrix} \conj{a[D]}&  \cdots & \conj{a[1]} & \conj{a[0]} \end{bmatrix}^{\top}.
\]
Finally, we define the multiplication of polynomials, as usual,  but viewing it as a map between the finite-dimensional vector spaces
\begin{equation}
\begin{split}
\bbC_{\le D_1}[z] \times \bbC_{\le D_2}[z] &\to \bbC_{\le (D_{1}+D_{2})}[z] \\
  (A(z), B(z)) & \mapsto C(z) = A(z)B(z).
\end{split}\label{def:mult}
\end{equation}
\begin{remark}
The multiplication of polynomials defined in \eqref{def:mult} (via the isomorphism \eqref{eq:poly})  becomes a bilinear mapping $\bbC^{D_1+1} \times   \bbC^{D_2+1} \to \bbC^{D_1+D_2+1}$, defined as
\begin{equation}\label{eq:polyMultMatrixProduct}
  (\bfa, \bfb)  \mapsto  \bfc =  \multmat{b}{D_1} \bfa=  \multmat{a}{D_2} \bfb,
\end{equation}
where $\multmat{a}{L}$ is the multiplication matrix 
\begin{equation}\label{eq:multiplicationMatrix}
  \multmat{a}{L} :=
  {\begin{bmatrix}
      \pco{a}{0} &        &            \\
      \vdots     & \ddots &            \\
      \pco{a}{D} &        & \pco{a}{0} \\
                 & \ddots & \vdots     \\
                 &        & \pco{a}{D}
    \end{bmatrix}} \in \mathbb{C}^{(D+ L +1) \times (L+1)},
\end{equation}
defined for any non-negative integer $L$ and a 
vector of coefficients $\bfa$ in \eqref{eq:poly}. 
Therefore, the multiplication of polynomials corresponds essentially to the convolution of vectors.
\end{remark}

We will also remark how the usual inner product on $\bbC^{N}$ can computed using multiplications of polynomials.
\begin{lemma}\label{lem:inner_product}
Let $X, Y \in \bbC_{\le N-1}[z]$.
The coefficient of the polynomial $X(z) \widetilde{Y}(z)$ at the monomial $z^{N-1}$ is equal to the inner product between the vectors of coefficients, that is:
\[
\left.(X(z) \widetilde{Y}(z))\right|_{z^{N-1}} = \langle X, Y \rangle = \sum\limits_{n=0}^{N-1} x[n] \conj{y[n]} = \langle \bfx, \bfy\rangle_{\bbC^N}.
\]
In particular, $\left.(X(z) \widetilde{X}(z))\right|_{z^{N-1}} = \langle X, X \rangle = \|X\|^2_2 = \Vert \bfx \Vert_2^2$.
\end{lemma}
\begin{proof}
It follows from straightforward calculation (see also Appendix \ref{app:autocor}, where the relation between \eqref{prob:PAF} and correlation sequences is detailed).
\end{proof}

\subsection{Divisors and greatest common divisors}
In this paper, we work with the multiplication defined in \Cref{def:mult}: as a result, we always take care of the space where polynomials belongs to\footnote{Again, we could work instead homogeneous bivariate polynomials, but we stick to the notation with univariate polynomials used in this paper.}.
Therefore, we need a special definition of divisibility.
\begin{definition}[Divisors of polynomials]\label{def:divi}
We say that the polynomial $C \in\bbC_{\le N}[z]$ has a divisor $B \in  \bbC_{\le D}[z] \setminus \{ \boldsymbol{0}\}$ if there is a polynomial $C \in \bbC_{\le N-D}[z]$ such that $C(z) = A(z)B(z)$ in the sense of \Cref{def:mult}.
\end{definition}
\begin{example}\label{ex:simple_poly_divisibility}
Consider the following polynomial from  $\bbC_{\le 5}[z]$:
\begin{equation}
\label{eq:ex_poly}
A(z) = 0 \cdot z^5 + 0 \cdot z^4 + \frac{1}{2} z^3 + \frac{1}{2} z^2 - z  \in \bbC_{\le 5}[z],
\end{equation}
which has two zero leading coefficients.
Then the polynomial  $F(z) = 0 \cdot z^3 + 0 \cdot z^2 + z + 2  \in \bbC_{\le 3}[z]$ is a divisor of the polynomial $A(z)$ in \eqref{eq:ex_poly},
because $A(z) = F(z)\cdot (\frac{1}{2}(z-1)z)$ according to definition \Cref{def:mult}.
However, the polynomial $G(z) = 0 \cdot z^4 + 0 \cdot z^3 + 0 \cdot z^2 + z + 2  \in \bbC_{\le 4}[z]$ is not a divisor of $A(z)$, because there is no $H(z) \in \bbC_{\le 1}[z]$ such that $A(z) = G(z)H(z)$.
Note that $F(z)$ and $G(z)$ represent the same polynomial $z+2$ in $\bbC[z]$ (if we forget the space in which the polynomial is living).
\end{example}
\begin{remark}[Alternative definition of the divisibility]\label{rem:divisibility_alternative}
\Cref{ex:simple_poly_divisibility} suggests the following equivalent definition of divisibility: 
the polynomial $C \in\bbC_{\le N}[z]$ has a divisor $B \in  \bbC_{\le D}[z] \setminus \{ \boldsymbol{0}\}$  if $B$ is a divisor of $C$ in the usual sense, and, in addition $B$ has at most the same number of zero leading coefficients as $C$.

For instance, in \Cref{ex:simple_poly_divisibility} the polynomial $z+2$ is a divisor of $A(z) = \frac{1}{2} z^3 + \frac{1}{2} z^2 - z  \in \bbC_{\le 5}[z]$, but $A(z)$ has two zero leading coefficients. The polynomials $F(z)$ and $G(z)$, in their turn, have $2$ and $3$ zero leading coefficients, respectively.
\end{remark}

Now we are ready to introduce the notion of the greatest common divisor of polynomials.
\begin{definition}\label{def:gcd}
For a set of polynomials, $X_1 \in \bbC_{\le N_1}[z], \ldots, X_K \in \bbC_{\le N_K}[z]$, the greatest common divisor $H$ is defined as the polynomial $H \in \bbC_{\le D}[z]$ with highest possible $D$, such that $H(z)$ is a divisor of all polynomials $X_1,\ldots,X_K$  in the sense of \Cref{def:divi}. 
We set $\gcd\{A,0\} = A$, and, if all polynomials are zero, then we formally set $ \gcd \{X_1,\ldots,X_K\} = 0 \in \bbC_{\le 0}[z] = \bbC$.
\end{definition}
The GCD of polynomials in \Cref{def:gcd} enjoys the usual properties of the greatest common divisor (such as associativity). In particular, the GCD exists and is unique up to a multiplication by a nonzero scalar. Therefore the notation
\begin{equation}\label{eq:gcd}
H = \gcd \{X_1,\ldots,X_K\}
\end{equation}
means that $H$ is a GCD;  the equality \eqref{eq:gcd} is meant modulo multiplication by a nonzero scalar (a typical abuse of notation in the literature). 
We say that the polynomials are coprime if $\gcd \{X_1,\ldots,X_K\} = 1$. 
\begin{remark}
Similarly to \Cref{rem:divisibility_alternative}, the GCD of polynomials $X_k \in \bbC_{\le N_k}[z]$  coincides with the GCD in the usual sense with the  zero leading coefficients equal to the minimal number of zero leading coefficients among $X_k$.
\end{remark}
Finally, we recall a matrix-based criterion of coprimeness of two polynomials.
\begin{theorem}\label{thm:sylvester}
Two polynomials $X_1, X_2 \in \bbC_{\le N-1}[z]$ are
coprime in the sense of  \Cref{def:gcd} if and only if its $2N \times 2N$ Sylvester matrix 
\[
S(X_1,X_2) = \begin{bmatrix} \multmat{x_1}{N-1} & \multmat{x_2}{N-1} \end{bmatrix}
\]
is nonsingular.
\end{theorem} 

We refer the reader to \cite[Section 2]{usevich_variable_2017} for more details and properties of the GCD, as well as the generalizations of \Cref{thm:sylvester} and the matrix-based algorithms to compute the GCD.

\subsection{Roots at infinity}
An important tool used in this paper is that we operate with $\infty$ roots.
We will say that the polynomial $A \in  \bbC_{\le D}[z]$ in \eqref{eq:poly} has a root at $\infty$ (with multiplicity $\mu_k$) if its leading $\mu_k$ coefficients vanishes (\ie if $a[D]=\cdots =a[D-\mu_k+1]=0$).
We will formally write $(z - \infty)^{d} B(z)$ to denote that $d$ zero leading coefficients are appended to the polynomial.
\begin{example}\label{ex:simple_poly}
Consider the polynomial from \Cref{ex:simple_poly_divisibility}.
The polynomial   has roots $\{\infty, -2, 1, 0\}$, where  the root $\infty$ has multiplicity $2$.
Hence it has the following factorization
\begin{equation}
\label{eq:ex_poly_factoriza}
A(z) = \frac{1}{2} (z-\infty)^2 (z - 1) (z+2) z.
\end{equation}
\end{example}
\begin{remark}
The root at infinity can be formally defined as:
\[
  (z-\infty) := 0 \cdot z + 1  \in \bbC_{\le 1} [z].
\]
Then the multiplication by such polynomial in the sense of the definition in \eqref{eq:polyMultMatrixProduct} corresponds exactly to adding a zero leading coefficient. 
In particular,
\[
(z-\infty)^{d} := 0 \cdot z^d +  0 \cdot z^{d-1}  +\cdots +  0 \cdot z + 1  \in \bbC_{\le d} [z].
\]
\end{remark}
With such a convention, the following extended version of the fundamental theorem of algebra holds true:
  any nonzero polynomial $A \in \bbC_{\le D}[z] \setminus \{ \boldsymbol{0}\}$ can be uniquely factorized  (up to permutation of roots)  as
  \begin{equation}\label{eq:polyFactorizationFTA}
    A(z) = \lambda \prod_{i=1}^{m}  (z-\alpha_{i})^{\mu_i},
  \end{equation}
  where $\lambda \in \bbC$, $\alpha_i \in \bbC \cup \{ \infty \}$ are distinct roots  and $\mu_i$ are the multiplicities of $\alpha_i$,
  so that their sum  is
  \[
    \mu_1 + \cdots +\mu_m =D.
  \]
Finally, we remark that the conjugate reflection \eqref{eq:conjReverse} leads to reflection of roots.
\begin{lemma}\label{lem:conj_reflection_roots}
The conjugate reflection of the polynomial \eqref{eq:polyFactorizationFTA} admits a factorization
  \[
  \widetilde{A}(z) = \widetilde{\lambda} \prod_{i=1}^{m}  \left(z-\overline{\alpha^{-1}_{i}}\right)^{\mu_i}, \quad \text{where } \widetilde{\lambda} := \overline{\lambda}  \prod_{\substack{i =1\\ \alpha_i \neq \infty}}^m (-\overline{\alpha_{i}})^{\mu_i},
  \]
  \ie the roots $\alpha_i$ are mapped to $\overline{\alpha^{-1}_i}$, where $0$ is formally assumed to be the inverse of $\infty$ and vice versa.
\end{lemma}
\begin{proof}
Follows from straightforward calculation.
\end{proof}

  \begin{example}
  For \Cref{ex:simple_poly}, the conjugate reflection $\widetilde{A}(z)  \in \bbC_{\le 5}[z]$, as well as its factorization  becomes:
  \[
  \widetilde{A}(z)  = 0 \cdot z^5 -  z^4 + \frac{1}{2} z^3 + \frac{1}{2} z^2  = (-1) (z - \infty) (z+\frac{1}{2})(z-1)z^2.
  \]
  which  has roots $\{\infty, 1, -\frac{1}{2}, 0\}$, where the root $0$ has multiplicity $2$.
  \end{example}
  Graphically,  the conjugate reflection of the roots has a nice interpretation in terms of the Riemann sphere: the mapping of the root under conjugate reflection becomes simply a reflection with respect to the plane passing through the equator, see Fig.~\ref{fig:RiemannSphere}.
  
  \begin{figure*}[t]
    \includegraphics[width=\textwidth]{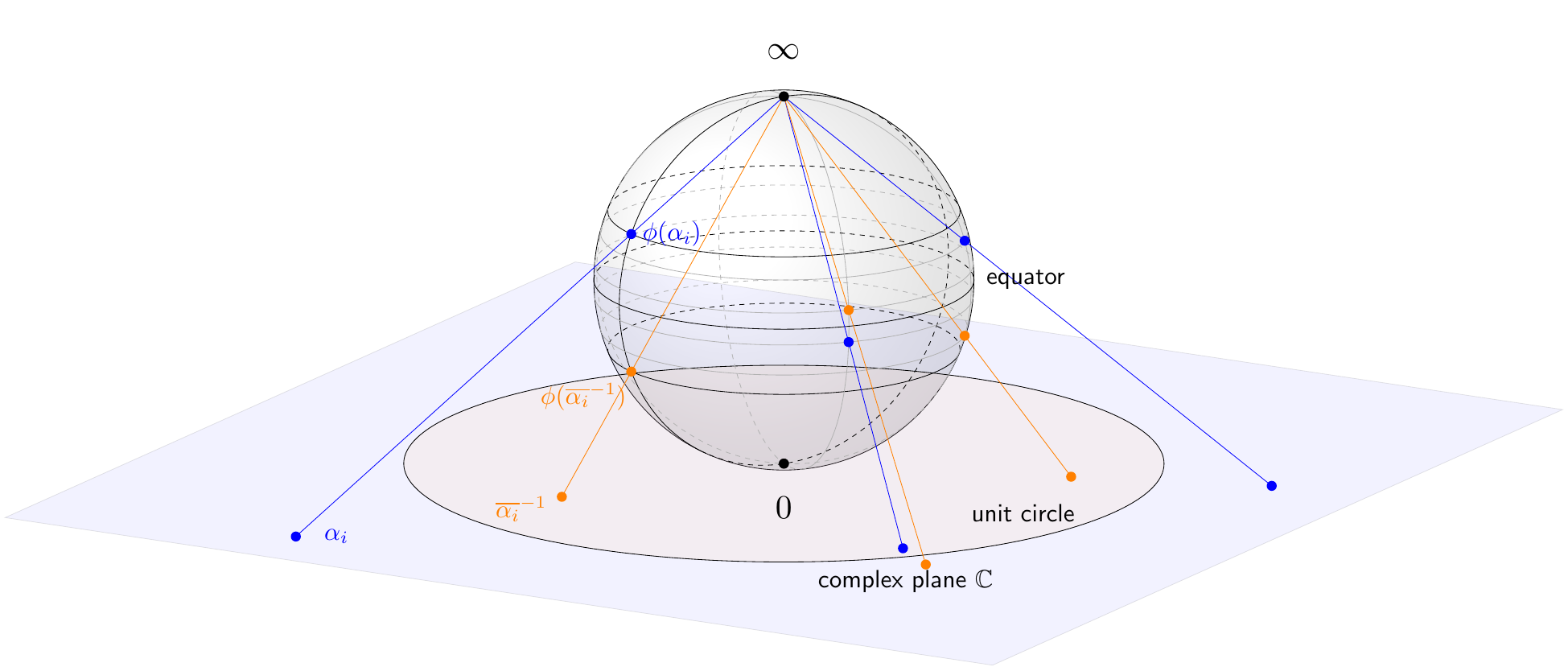}
    \caption{Complex plane and the Riemann sphere (the preimage under the stereographic projection). The conjugate inversion corresponds to reflection with respect to the equator on the Riemann sphere. Here $\phi: \bbC \to \calS^2$ denote the inverse stereographic mapping onto the sphere $\calS^2$.}\label{fig:RiemannSphere}
  \end{figure*}

\begin{remark}
When dealing with homogeneous polynomials, the roots in fact belong to the projective space $\mathbb{P}^1$ which corresponds exactly to $\mathbb{C} \cup \{\infty\}$.
\end{remark}

\begin{remark}
Let \eqref{eq:polyFactorizationFTA} be the factorization of a polynomial  $A \in\bbC_{\le D}[z]$ with roots $\alpha_i$ of respective multiplicities $\mu_i$. 
Then $B\in  \bbC_{\le D'}[z] \setminus \{ \boldsymbol{0}\}$ is a divisor of $A$ if and only if it can be factorized as 
 \begin{equation}\label{eq:polyFactorizationFTAbis}
    B(z) = \lambda' \prod_{i=1}^{m}  (z-\alpha_{i})^{\nu_i},
  \end{equation}
where  $0\leq \nu_i \leq \mu_i$ and $\nu_1 + \cdots + \nu_m = D'$.
\end{remark}

\begin{example}\label{ex:simple_poly_divisibility_roots}
Continuing  \Cref{ex:simple_poly_divisibility}, the polynomial $0 \cdot z^4 + 0 \cdot z^3 + 0 \cdot z^2 + z + 2 = (z-\infty)^3(z+2) \in \bbC_{\le 4}[z]$ is not the divisor of $A(z)$  from \eqref{eq:ex_poly}, because there are not enough infinite roots in the expansion of $A(z)$.
\end{example}

\section{Uniqueness of factorizations}
  \label{sub:uniquenessGeneral}
The main goal of this section is to provide a proof of  \Cref{eq:thm_main}, thus giving a characterization of the uniqueness properties of the polynomial factorization problem \eqref{prob:PAF}.
In fact, we will prove a generalized version of \Cref{eq:thm_main}.


\subsection{Key lemma and the coprime case}
The following key lemma links the GCD of polynomials $X_k(z)$ with the GCD of the elements of the matrix polynomial $\boldsymbol\Gamma(z)$.
  \begin{lemma}\label{lemma:PAFspectralFactorisation}
Let $Q(z) := \gcd\lbrace X_1, X_2, \ldots, X_K\rbrace$  where $Q \in \bbC_{\le D}[z]$, and define
\begin{equation}\label{eq:H_and_Q}
H(z)  = Q(z)\tilde{Q}(z).
\end{equation}
Then the GCD of the elements of $\boldsymbol{\Gamma}(z)$ must
be equal to $H(z)$:
\[
 \gcd\lbrace\Gamma_{ij}\rbrace_{i,j=1}^{K,K} = H(z).
\]
\end{lemma}
\begin{proof}
Let $R_1, R_2, \ldots R_K \in \bbC_{\le N-D-1}[z]$  be the corresponding quotients, \ie $X_k(z) = Q(z)R_k(z)$ for $k = 1, \ldots, K$ with $\gcd\lbrace R_1,R_2, \ldots, R_K\rbrace = 1$.
Direct calculations show that, for $i,j = 1, \ldots, K$,  
\[
\Gamma_{ij}(z)  = X_i(z)\tilde{X}_j(z)= R_i(z)Q(z)\tilde{R}_j(z)\tilde{Q}(z) = R_i(z)\tilde{R}_j(z) H(z) .
\]
  Then the GCD of polynomials $\Gamma_{ij}(z)$ can be explicitly computed as
    \begin{align*}
      &  \gcd\lbrace\Gamma_{ij}\rbrace_{i,j=1}^{K,K}\\
       & = \gcd\left\lbrace\gcd\lbrace\Gamma_{1j}\rbrace_{j=1}^K,\gcd\lbrace\Gamma_{2j}\rbrace_{j=1}^K, \ldots, \gcd\lbrace\Gamma_{Rj}\rbrace_{j=1}^K\right\rbrace &                                                   \\
    & = \gcd\left\lbrace R_1H, R_2H, \ldots,R_K H \right\rbrace                                      & \text{ since } \gcd\lbrace \tilde{R_1}, \tilde{R}_2, \ldots, \tilde{R}_K\rbrace = 1 \\
    & = H(z)                                                       & \text{ since } \gcd\left\lbrace R_1, R_2, \ldots R_K\right\rbrace = 1,
    \end{align*}
which concludes the proof.
\end{proof}
\begin{remark}\label{rem:essential_uniqueness}
If the matrix polynomial $\boldsymbol{\Gamma}(z)$ can be factorized as \eqref{prob:PAF}, then any simultaneous rescaling of the polynomials $X_1(z), \ldots, X_K(z)$ by $\beta\in \mathbb{T}$ provides an alternative factorization since 
\[
\begin{bmatrix}
\beta X_1(z) \\
\vdots \\
\beta X_K(z) 
\end{bmatrix}
\begin{bmatrix}
\conj{\beta} \widetilde{X}_1(z) &
\cdots &
\conj{\beta}  \widetilde{X}_K(z) 
\end{bmatrix} = \begin{bmatrix}
X_1(z) \\
\vdots \\
X_K(z) 
\end{bmatrix}
\begin{bmatrix}
\widetilde{X}_1(z) &
\cdots &
\widetilde{X}_K(z) 
\end{bmatrix},
\]
i.e., polynomials $Y_k(z) = \beta X_k(z)$ provide an equivalent factorization since $\beta \conj{\beta} = 1$.

In what follows, we will say that the solution \eqref{prob:PAF} is essentially unique if it is uniqueness up to a global scaling by $\beta \in \bbT$, \ie for any alternative factorization 
\begin{equation}\label{eq:paf_alternative}
\Gamma(z) = \begin{bmatrix} Y_1(z) \\ \vdots \\ Y_K(z)  \end{bmatrix}
\begin{bmatrix} \widetilde{Y}_1(z) & \cdots & \widetilde{Y}_K(z) \end{bmatrix},
\end{equation}
there exists $\beta \in \mathbb{T}$ such that $Y_k(z) = \beta X_k(z)$ for all $k$.
\end{remark}

Armed with \Cref{lemma:PAFspectralFactorisation}, we can already give the characterization of the coprime case.
\begin{proposition}\label{prop:coprime}
Let $\{\Gamma_{ij}\}_{i,j=1}^{K,K}$ be coprime (equivalently, let $X_1(z),\ldots,X_K(z)$ be coprime). Then
\begin{enumerate}
\item the rank-one factorization \eqref{prob:PAF} is  essentially unique (in the sense of  \cref{rem:essential_uniqueness});
\item in particular, let  
\begin{equation}\label{eq:Gamma_row_gcd}
A_j(z) :=  \gcd(\Gamma_{j1}, \ldots, \Gamma_{jK})
\end{equation}
and fix an index $j$ such that $A_j(z) \not\equiv 0$. Then a factorization \eqref{eq:paf_alternative} can be obtained by 
\begin{equation}\label{eq:paf_alternative_constructive}
Y_k(z) =c_j \cdot \frac{\Gamma_{kj}(z)}{\widetilde{A}_j(z)},
\end{equation}
where $c_j$ is the normalization constant $c_j = {\|A_{j}\|}/{\sqrt{\Gamma_{jj}(z)|_{z^N-1}}}$.
\end{enumerate}
\end{proposition}
\begin{proof}
We first prove 2).
We first show that the polynomial $Y_k(z)$ defined in \eqref{eq:paf_alternative_constructive} provides an alternative factorization.
Indeed, for $A_j$ defined in \eqref{eq:Gamma_row_gcd} we have
\begin{equation}
A_j(z) =  \gcd(X_{j} \widetilde{X}_1, \ldots, X_{j} \widetilde{X}_K)= d_j X_j(z), \quad d_j \neq 0.
\end{equation}
since $\widetilde{X}_1(z),\ldots, \widetilde{X}_{K}(z)$ are coprime.
Then for the polynomial $Y_k(z)$ defined in \eqref{eq:paf_alternative_constructive} we get
\begin{equation}\label{eq:Yk_Xk_relation}
Y_k(z) = \frac{X_k(z)\widetilde{X}_j(z)}{\overline{d_j} \widetilde{X}_j(z)}
\cdot \frac{\|d_j X_j(z)\|}{\|X_j(z)\|} =  {\frac{d_j}{|d_j|}}X_k(z),
\end{equation}
thanks to \Cref{lem:inner_product}. By defining $\beta = \frac{d_j}{|d_j|}$, we observe that  $\beta \in \bbT$, and therefore $\Gamma_{ij}(z) = Y_{i}(z) \widetilde{Y}_j(z)$.

\noindent{Proof of 1)} Assume that apart from \eqref{eq:paf_alternative}, there exists an alternative factorization $\Gamma_{ij}(z) = U_{i}(z) \widetilde{U}_j(z)$.
Note that by \Cref{lemma:PAFspectralFactorisation}, we have $\gcd\{U_1,\ldots,U_K\} = 1$.
Therefore $\Gamma_{ij}(z) = X_{i}(z) \widetilde{X}_j(z)$  and $\Gamma_{ij}(z) = U_{i}(z) \widetilde{U}_j(z)$ are two valid factorizations that satisfy the conditions of 2).
Now assume that $Y_k$ is computed as in \eqref{eq:Yk_Xk_relation}.
Then, by 2) there are two constants $\beta,\delta \in \bbT$ such that 
\[
Y_k=\beta X_k \quad \text{and}\quad  Y_k=\delta U_k, \quad \text{for all }k. 
\]
Therefore,   we have that $U_k = \delta^{-1} \beta X_k$, for all $k$, which completes the proof.
\end{proof}

\Cref{prop:coprime} already tells us how to deal with the coprime case, and provides us  with a constructive way to find the factorization  \eqref{prob:PAF} from the matrix polynomial $\boldsymbol{\Gamma}(z)$ (\ie, by computing $Y_k$ as in \eqref{eq:Yk_Xk_relation}).
Moreover, this algorithm can be further simplified as shown by the following remark.
\begin{remark}
In the conditions of \eqref{prop:coprime}, the polynomials $X_k$ can be retrieved, up to individual constants from \eqref{eq:Gamma_row_gcd} (\i.e. $A_k  = \beta_k X_k$).
Therefore, if one is interested, for example, only in the roots of $X_k$, those can be obtained just by computing $A_k$ as in \eqref{eq:Gamma_row_gcd}.
\end{remark}
We also note that the \Cref{prop:coprime} covers the generic case, as shown by the following corollary.

\begin{corollary}[Almost everywhere uniqueness of \eqref{prob:PAF}]
In the generic case, the solution of \eqref{prob:PAF} is essentially unique:
there exists a set of (Lebesgue) measure zero in $A \subset (\bbC_{\le N -1}[z])^K$, such that for all $(X_1,\ldots,X_K) \in (\bbC_{\le N-1 }[z])^K \setminus A$ the solution of \eqref{prob:PAF} is essentially unique.
\end{corollary}
\begin{proof}
By \Cref{thm:sylvester} the two polynomials $X_1, X_2 \in \bbC_{\le N-1}[z]$ are
coprime if and only if $S(X_1,X_2)$ is invertible.
The equation $\det(S_1(X_1,X_2)) = 0$ defines an algebraic variety $V$ of dimension   $2N-1\le \dim((\bbC_{\le N-1}[z])^{2})$, thus it is a set of measure zero.
Taking $A = V \times\dim(\bbC_{\le N-1}[z])^{K-2}$ concludes the proof.
\end{proof}

Finally, we remark that the condition of coprimeness in \Cref{prop:coprime} is only a sufficient condition (and not a necessary condition as mistakenly claimed  in \cite[Theorem 1]{raz_vectorial_2013}). We establish a necessary and sufficient condition in the next subsection.

\subsection{The main uniqueness result}
Before proving the main uniqueness result, we establish another important lemma that shows what happens in the non-coprime case.

\begin{lemma}\label{lem:alternativeFactorizations}
Let $X_1(z), \ldots, X_K(z)$ be a tuple  of polynomials not vanishing simultaneously. Let $Q = \gcd \{X_1,\ldots,X_K\} \in \bbC_{\le D}[z]$ and $R_1, R_2, \ldots R_K \in \bbC_{\le N-D-1}[z]$  be the corresponding quotients such that $X_k(z) = Q(z) R_k(z)$ for any $k$.
Then $Y_k(z)$ provides a valid  alternative factorization \eqref{eq:paf_alternative} if and only if the polynomials have the form
\begin{equation}\label{eq:alternative_factors}
Y_k(z) =  S(z) R_k(z)
\end{equation}
where $S(z)$ satisfies $S(z)\widetilde{S}(z) = Q(z) \widetilde{Q}(z)$.
\end{lemma}
\begin{proof}
The ``if'' part is obvious, because for any polynomial of the form \eqref{eq:alternative_factors}
\[
Y_i(z) \widetilde{Y}_j(z) = S(z)\widetilde{S}(z)R_i(z)\widetilde{R}_j(z) = Q(z) \widetilde{Q}(z)R_i(z)\widetilde{R}_j(z) = X_i(z) \widetilde{X}_j(z).
\]
Now we will prove the ``only if'' part.

Let $H(z) = Q(z) \widetilde{Q}(z)$ (so that $H(z) = \gcd\lbrace\Gamma_{ij}\rbrace_{i,j=1}^{K,K}$ by \Cref{lemma:PAFspectralFactorisation}) , and fix and alternative factorization $Y_k(z)$. 
Then, from \Cref{lemma:PAFspectralFactorisation}, a GCD $S = \gcd\{Y_1,\ldots,Y_K\}$ must satisfy $S(z)\widetilde{S}(z) = c H(z)$ with $c \in \bbC$, where we can normalize $c$ to be $1$ (thanks to the freedom of choosing the normalization for $S(z)$).

Now denote by $T_k(z)$ the corresponding quotients of $Y_k(z)$ such that $Y_k(z) = S(z) T_k(z)$.
Then the matrix polynomial
\[
\boldsymbol{G}(z) := \frac{\boldsymbol{\Gamma}(z)}{H(z)} 
\]
must have the factorization of its entries as $G_{ij}(z) = R_i(z) \widetilde{R}_j(z) = T_i(z) \widetilde{T}_j(z)$.
However, since $H = \gcd \{\Gamma_{ij}\}_{i,j=1}^{K,K}$, we have $\gcd \{G_{ij}\}_{i,j=1}^{K,K} = 1$, and therefore, by \Cref{prop:coprime} it has an essentially unique factorization.
Thus there exists a constant $\beta \in \bbT$ such that $T_k = \beta R_k$ for all $k$. 
In turn, we can define $S'(z) := \beta S(z)$, and one gets that $Y_k(z)$ can be expressed as
\[
Y_k(z) =  S'(z) R_k(z),
\]
where $S'(z)\widetilde{S'}(z) = Q(z)\widetilde{Q}(z)$,
which completes the proof.
\end{proof}

\begin{remark}\label{rem:uniqueness_H}
\Cref{lem:alternativeFactorizations} shows that the study of the uniqueness properties of \eqref{prob:PAF} is directly related to uniqueness of the univariate polynomial factorization $H(z)= Q(z)\widetilde{Q}(z)$, as the quotients can be obtained thanks to the constructive procedure described in \Cref{prop:coprime} applied to $\frac{\boldsymbol{\Gamma}(z)}{H(z)}$.

In other words, all possible factorizations can be obtained from the Smith normal form of the rank-one matrix polynomial $\boldsymbol{\Gamma}(z)$ \cite{mackey2011smith}.
\end{remark}

  Before giving the sufficient and necessary uniqueness condition, we make a remark about the roots of the product $Q(z)\widetilde{Q}(z)$ which are key to understanding uniqueness.
  
\begin{lemma}\label{rem:roots_H}
Let $Q(z) = \lambda \prod_{i=1}^D (z-\alpha_i)$ (with possibly repeating $\alpha_i$). Then $H(z) = Q(z)\tilde{Q}(z)$ has the following factorization
\begin{equation}
H(z) =\lambda \widetilde{\lambda} \prod_{i: \alpha_i \neq \infty}(-\overline{\alpha_i})  \prod_{i=1}^D (z-\alpha_i) (z- \overline{\alpha_i^{-1}}).
\end{equation}
Furthermore, if $\alpha \in \bbT$, then $\alpha = \overline{\alpha^{-1}}$.
Therefore, a unit-modulus $\alpha$ is a root of $Q(z)$ of multiplicity $\mu$ if and only if it is a root of $H(z)$ of multiplicity $2\mu$.
\end{lemma}
\begin{proof}
Follows from \Cref{lem:conj_reflection_roots}.
\end{proof}

\begin{theorem}\label{corr:uniquenessPAF}
Let $X_1(z), \ldots, X_K(z)$ be a tuple  of polynomials not vanishing simultaneously.
Then the following equivalences are true:
\begin{enumerate}
\item The problem \eqref{prob:PAF} admits a unique solution (in the sense of \Cref{rem:essential_uniqueness});
\item $X_1(z), \ldots, X_K(z)$ have no common roots in  $(\mathbb{C} \cup \{\infty \}) \setminus \bbT$  (common roots may be only on the unit circle);
\item $H(z) =  \gcd\lbrace\Gamma_{ij}\rbrace_{i,j=1}^{K, K}$ has no roots in  $\mathbb{C} \setminus \bbT$.
\end{enumerate}
\end{theorem}

\begin{proof} 
The proof is organized in several parts.
\begin{itemize}
\item  $\boxed{2 \Leftrightarrow 3}$ By \Cref{lemma:PAFspectralFactorisation}, $H(z) = cQ(z) \widetilde{Q}(z)$, where $c$ is a constant and $Q(z) = \gcd\lbrace X_i (z)\rbrace_{i=1}^{K}$. Therefore, by \Cref{rem:roots_H}, $H(z)$ does not have roots outside the unit circle $\bbT$ if and only if $Q(z)$ does not.
Note that by  \Cref{rem:roots_H}, the roots of $H(z)$ appear in pairs, and therefore $H(z)$ has an $\infty$ root if and only $0$ is also a root.

  \item  $\boxed{1 \Rightarrow 2}$ Suppose that the solution of \eqref{prob:PAF} is essentially unique, but the polynomial $Q(z)$ has a root $\alpha$ outside the unit circle.
  Then easy calculations show that polynomial $S(z) = \frac{Q(z)(z-\overline{\alpha^{-1}})}{(z-\alpha)}$ satisfies
  \[
  S(z)\widetilde{S}(z) = Q(z) \widetilde{Q}(z).
  \]
Note that $S(z)$ is not proportional to $Q(z)$, because
\[
\frac{(z-\overline{\alpha^{-1}})}{(z-\alpha)} \neq \text{const}.
\]
Therefore the vector polynomial
\[
(Y_1(z),\ldots, Y_K(z))  :=  (S(z)R_1(z), \ldots, S(z) R_K(z)),
\]
is not proportional to the vector $(X_1(z),\ldots,X_K(z))$, but gives an alternative factorization $\Gamma_{ij}(z) = Y_i(z) \widetilde{Y}_j(z)$ (a contradiction).
  
  \item $\boxed{1 \Leftarrow 2}$ 
Let $Q(z)$ be the GCD of $\{X_k\}_{k=1}^K$ and $R_k(z)$  be the corresponding quotients.
Since  $Q(z)$ has only unit-modulus roots, by \cref{rem:roots_H}, the polynomial $H(z) = Q(z) \widetilde{Q}(z)$ has the roots with doubled multiplicities.
Therefore, there is a unique (up to a constant $\beta \in \bbT$) way to factorize $H(z)$ (for any alternative factorization, $H(z) = S(z) \widetilde{S}(z)$, there exists a constant $\beta \in \bbT$ such that $S(z) = \beta Q(z)$).
Therefore, by \Cref{lem:alternativeFactorizations}, any other valid alternative factorization $Y_k(z)$ has necessarily the form
\[
Y_k(z) = \beta Q(z)R_k(z),
\]
which completes the proof.
\end{itemize}
\end{proof}


%

\section{Enumerating the factorizations}\label{sec:nonuniqueness}
In this section, we refine \cref{corr:uniquenessPAF} by providing the number of solutions and by describing the set of solutions of  \eqref{prob:PAF} in the non-unique case.
As mentioned in \Cref{rem:uniqueness_H}, this description depends mainly on  uniqueness properties of the factorization \eqref{eq:H_and_Q} (i.e., how to find all $Q(z)$ such that $H(z) = Q(z)\tilde{Q}(z)$).
The univariate factorization problem, in its turn,  is known to  closely related  to enumerating the solutions of the so-called univariate phase retrieval problem  \cite{beinert2015ambiguities,boche_fourier_2017}.
In this section we enumerate in \cref{thm:enumerating_solutions}  all possible factorizations based on the technique used for the phase retrieval problem  (see \cite[Theorem 3.1, Corollary 3.3]{beinert2015ambiguities} and \cite[Proposition 6.1]{beinert2016ambiguities}) which goes back to the result of Fej\'{e}r
\cite[p.61]{fejer1916trigonometrische}.


The number of ways we can factorize $H(z)$ depends on multiplicities of its roots, and the following remark is very  useful.
\begin{remark}[Root pairs of $H(z)$]%
From \Cref{rem:roots_H}, we know that the roots of $H(z)$ come in conjugate-reflected pairs, \ie if a root $\delta \not\in \bbT$ is a root of $H(z)$ with multiplicity $\mu$ then  $\conj{\delta}^{-1}$  is also a root with the same multiplicity.
In such a case we will say that the pair $(\delta, \conj{\delta}^{-1})$ has multiplicity $\mu$.
In the case where $\delta\in \bbT$ is a root of $H(z)$, then it must have even multiplicity.
\end{remark} 
  \begin{example}\label{ex:root_pairs}
  Consider $Q(z) =A(z)$ that is the polynomial from \cref{ex:simple_poly} having double $\infty$ root and simple roots $\{-2,1,0\}$. 
  Then the polynomial $H(z) = A(z) \tilde{A}(z)$ is given by
  \[
  H(z) = -\frac{1}{2} (z - \infty)^{3} \left(z+\frac{1}{2}\right)(z+2) (z-1)^2 z^3.
  \]
  The multiplicity of the root pair $(0,\infty)$ is $3$, the root pair $(-2,-\frac{1}{2})$ has multiplicity $1$, and the root pair $1$ has multiplicity $2$. 
  \end{example}

\begin{remark}
As shown in \Cref{ex:root_pairs}, different roots of $Q(z)$ can merge to form root pairs (as it is the case for the roots $-2$ and $-\frac{1}{2}$ in \Cref{ex:root_pairs}). In general, if  $\delta$ and $\conj{\delta}^{-1}$ are roots of $Q(z)$ with multiplicities $\nu_1$ and $\nu_2$, then the multiplicity of the pair $(\delta, \conj{\delta}^{-1})$ of $H(z)$ is equal to  $\nu_1 + \nu_2$.
\end{remark}


\begin{theorem}[Enumerating solutions of \eqref{prob:PAF}]\label{thm:enumerating_solutions}
Let $Q, R_1,\ldots,R_K$ be as in \Cref{lem:alternativeFactorizations},
where the root structure of the polynomial $H(z) = Q(z)\tilde{Q}(z)$ is as follows:
\begin{itemize}
\item $H(z)$ has $P$ root pairs
$(\delta_1, \conj{\delta}_1^{-1}), \ldots, (\delta_{P}, \conj{\delta}_{P}^{-1})$ outside the unit circle with multiplicities $\mu_1,\ldots,\mu_{P}$ and 
\item roots $\varepsilon_1,\ldots,\varepsilon_T \in \bbT$ with (even) multiplicities $\nu_1,\ldots,\nu_T$.
\end{itemize}
Then all the possible alternative factorizations $Y(z)$ such that $\boldsymbol{\Gamma}(z) = Y(z)Y(z)$ may be expressed as
    \begin{align*}
    Y_1(z) &= \beta S(z) R_1(z)\\
    \vdots &\\
    Y_K (z) &= \beta S(z) R_K(z)\\
  \end{align*}
where $\beta \in \bbT$ is arbitrary, and where
\begin{equation}\label{eq:all_H_factorizations}
S(z) = \lambda \prod\limits_{i=1}^{P} (z-\delta_i)^{\ell_i} (z- \conj{\delta}_{i}^{-1})^{\mu_i - \ell_i}
        \ \prod\limits_{j=1}^{T} (z - \varepsilon_j)^{\frac{\nu_j}{2}},
\end{equation}
 $0\le \ell_i\le \mu_i$ are nonnegative integers, and $\lambda > 0$ is a constant that depends on a particular collection of $\ell_i$ (see \Cref{rem:roots_H}).
\end{theorem}
\begin{remark}
\Cref{thm:enumerating_solutions} gives a way to obtain all the polynomials from $\boldsymbol{\Gamma}(z)$. Indeed, the factorization relies on $H(z) = \gcd \{\Gamma(z)\}_{i,j=1}^{K}$ (see \Cref{lemma:PAFspectralFactorisation}) and the quotient polynomials  $R_1(z), \ldots, R_K(z)$ can be also obtained from $\boldsymbol{\Gamma}(z)$  thanks to  \Cref{prop:coprime}.
\end{remark}

\begin{corollary}[Number of factorizations of \eqref{prob:PAF}]\label{theorem:PAFuniqueness}
Under the assumptions of \Cref{thm:enumerating_solutions}, the problem \eqref{prob:PAF} admits exactly
\begin{equation}\label{eq:numberSolutionsPAF}
\prod\limits_{i=1}^{P} (\mu_i + 1)
\end{equation}
different solutions.
In particular, when  roots of $H(z)$ are all simple and outside the unit circle, there is exactly $2^{D}$ different solutions.
\end{corollary}

\begin{proof}[Proof of \Cref{thm:enumerating_solutions}]
\Cref{lem:alternativeFactorizations} shows that the number of solutions of \eqref{prob:PAF} is exactly the number of different (up to multiplication by a scalar) polynomials $S(z)$ such that $H(z) = S(z)\tilde{S}(z)$.
    This \emph{spectral factorization} problem is equivalent to selecting the roots of $Q(z)$ amongst the root pairs $(\delta_i, \conj{\delta}_i^{-1})$ of $H(z)$ outside $\bbT$.
  
  Consider
   a non-unit-modulus root pair $(\delta_i, \conj{\delta}_i^{-1})$ with multiplicity $\mu_i$; then the number of different combinations is equal to the number of outcomes of a random draw of $\mu_i$ items with replacement in a set of $2$ elements, \ie $\mu_i +1$.
    Repeating the same process for each root pair gives all possible factorizations \eqref{eq:all_H_factorizations} by selecting the integers $0 \le \ell_i \le \mu_i$ (the multiplicity of $\delta_i$ in $S(z)$).
\end{proof}

  \begin{example}\label{ex:root_pairs_count}
Continuing \Cref{ex:root_pairs}, we see that there are two root pairs not on the unit circle (with multiplicities $3$ and $1$, respectively).
  This yields a total of $4\cdot 2= 8$ solutions, where the other factorizations are given by permuting $0$ and $\infty$ roots or/and replacing root $-2$ with $-\frac{1}{2}$,
  For example, some of possible alternative factorisations are given by $Q(z) = \frac{1}{2} (z+2) (z-1) z^3$ or $Q(z) = (z+\frac{1}{2})(z-1)z^3$.
  \end{example}

\subsection{Case of two polynomials}

  We conclude this paper by  providing an explicit expression of  solutions of \eqref{prob:PAF} in the simplified case of $K =2$ and where there are no $0$ or $\infty$ roots in common, meaning that $\bfx[0] \neq \mathbf{0}$ and $\bfx[N-1] \neq \mathbf{0}$.
This setting is relevant to the context of polarimetric phase retrieval \cite{flamant2023polarimetric}.

 \begin{proposition}
Let $K=2$, and $H(z)$ be with roots in $\bbC \setminus\{0,\infty\}$, and fix a factorization of $H(z)$:
\[
H(z) = c \sum\limits_{i=1}^D (z-\beta_{i})(z-\overline{\beta_i^{-1}}).
\]
Let $\{\alpha_{ji}\}_{i=1}^{N-D-1}$ be the roots (with repetitions) of the quotient polynomials $R_j$, $j\in \{1,2\}$. 
Then the corresponding rank-one factorization has the form
\begin{align}
       X_1(z) & = e^{\jmath \theta} \lambda_1 \prod_{i=1}^{D}(z-\beta_{i})\prod_{i=1}^{N-D-1}(z-\alpha_{1i}), \\
       X_2(z) & = e^{\jmath \theta} \lambda_2\prod_{i=1}^{D}(z-\beta_{i}) \prod_{i=1}^{N-D-1}(z-\alpha_{2i}),
\end{align}
where
the constants $\lambda_1, \lambda_2 \in \bbC$  are given by
    \begin{align}
      \lambda_1 & =\sqrt{ \left \vert \gamma_{11}[N-1]\right\vert \prod_{i=1}^{D} | \beta_{i}| ^{-1}\prod_{i=1}^{N-D-1} | \alpha_{1i}|^{-1}},               \\
      \lambda_2 & = e^{\jmath\Delta}\sqrt{ \left\vert \gamma_{22}[N-1]\right\vert \prod_{i=1}^{D} | \beta_{i}| ^{-1}\prod_{i=1}^{N-D-1} | \alpha_{2i}|^{-1}},
    \end{align}
    where $\Delta$ reads
    \begin{equation}
      \Delta =  \pi(N-1)+ \arg  \gamma_{12}[N-1] + \sum_{i=1}^D\arg \beta_i + \sum_{i=1}^{N-D-1}\arg \alpha_{2i}\:.
    \end{equation}
\end{proposition}
\begin{proof}
To determine $\lambda_1$ and $\lambda_2$, one writes the expression of the elements of matrix polynomials in terms of $X_1(z)$ and $X_2(z)$ above. For instance:
\begin{equation}
  \begin{split}
    \Gamma_{11}(z) &= X_1(z)z^{N-1}\overline{X_1(\overline{z}^{-1})} \\
    &= \vert \lambda_1\vert^2 \prod_{i=1}^D(z-\beta_{i})\prod_{i=1}^{N-D-1}(z-\alpha_{1i}) \prod_{i=1}^D(1-\overline{\beta_{i}}z)\prod_{i=1}^{N-D-1}(1-\overline{\alpha_{1i}}z)
  \end{split}
\end{equation}
Using that $\Gamma_{11}(z) := \sum_{n=0}^{2N-2} \gamma_{11}[n-N+1]z^n$, identifying leading order coefficients yields
\begin{equation}
  \gamma_{11}[N-1]= \vert \lambda_1\vert^2(-1)^{N-1}\prod_{i=1}^D\overline{\beta_{i}}\prod_{i=1}^{N-D-1}\overline{\alpha_{1i}}
\end{equation}
Similarly, one gets
\begin{align}
  \gamma_{22}[N-1]= \vert \lambda_2\vert^2(-1)^{N-1}\prod_{i=1}^D\overline{\beta_{i}}\prod_{i=1}^{N-D-1}\overline{\alpha_{2i}} \\
  \gamma_{12}[N-1] = \lambda_1\overline{\lambda}_2(-1)^{N-1}\prod_{i=1}^D\overline{\beta_{i}}\prod_{i=1}^{N-D-1}\overline{\alpha_{2i}}
\end{align}
These relations determine uniquely the moduli of $\lambda_1, \lambda_2$ as well as the difference of argument between $\lambda_1$ and $\lambda_2$.
Thus $\lambda_1, \lambda_2$ are unique up to a global phase factor $\exp(\bmj \theta), \: \theta \in [-\pi, \pi)$.
One obtains eventually the following expressions
\begin{align}
  \lambda_1 = e^{\bmj \theta}\left(|\gamma_{11}[N-1]| \prod_{i=1}^{D} | \beta_{i}| ^{-1}\prod_{i=1}^{N-D-1} | \alpha_{1i}|^{-1}\right)^{1/2} \\
  \lambda_2 = e^{\bmj (\theta - \Delta)}\left(|\gamma_{22}[N-1]| \prod_{i=1}^{D} | \beta_{i}| ^{-1}\prod_{i=1}^{N-D-1} | \alpha_{2i}|^{-1}\right)^{1/2}
\end{align}
with
\begin{align}
  \Delta & = \arg (\lambda_1\overline{\lambda_2})                                                                \\
         & = \pi(N-1)+ \arg  \gamma_{12}[N-1] + \sum_{i=1}^D\arg \beta_i + \sum_{i=1}^{N-D-1}\arg \alpha_{2i}\:.\qedhere
\end{align}

\end{proof}

\section*{Acknowledgments}
Authors acknowledge the support of the GdR ISIS / CNRS through the exploratory research grant OPENING.

\begin{appendices}

\section{Link between matrix polynomial factorization and autocorrelation}\label{app:autocor}
The matrix polynomial rank-one factorization problem \eqref{prob:PAF} arises in multivariates instances of Fourier phase retrieval \cite{flamant2023polarimetric} and blind multichannel system identification \cite{jaganathan_reconstruction_2019}. 
In such applications, one is interested in recovering a deterministic discrete vector signal $\bfx: \llbracket 0, N-1\rrbracket\rightarrow \bbC^R$ from the different cross-correlations functions between the $R$ signal channels. 
%
Now, define the polynomial representation of the $i$-th channel of $\bfx$ as $X_i(z) = \sum_{n=0}^{N-1} x_i[n] z^n$.
Similarly, define the correlation polynomial $\Gamma_{ij}(z) := \sum_{n=0}^{2(N-1)} \gamma_{ij}[n - N+1]z^n$.
Then, a key result is that 
\begin{align}
    \Gamma_{ij}(z) = X_i(z)\widetilde{X}_j(z)\label{eq:univariateFactor}
\end{align}
since
\begin{align}
    X_i(z)\widetilde{X}_j(z) &= \left(\sum_{n=0}^{N-1} x_i[n]z^n\right)\left(\sum_{m=0}^{N-1} \conj{x_j[N-1-m]}z^m\right)\\
    &= \sum_{m=0}^{N-1}\sum_{n=0}^{N-1} x_i[n]\conj{x_j[N-1-m]}z^{n+m}\\
    &=\sum_{n=0}^{N-1} \sum_{m=0}^{N-1} x_i[n]\conj{x_j[m]}z^{n+N-1-m}\\
    & = \sum_{n'=0}^{2(N-1)} \gamma_{ij}[n'-N+1]z^{n'}:= \Gamma_{ij}(z). 
\end{align}
Therefore, defining the matrix polynomial $\boldsymbol{\Gamma}(z)$ such that 
\begin{equation}
    \boldsymbol{\Gamma}(z) = \begin{bmatrix}
        \Gamma_{11}(z) & \cdots &\Gamma_{1R}(z)\\
        \vdots & & \vdots\\
        \Gamma_{R1}(z) & \cdots &\Gamma_{RR}(z)\\
    \end{bmatrix} = \sum_{n=0}^{2N} \begin{bmatrix}
        \gamma_{11}[n-N+1]& \cdots &\gamma_{1R}[n-N+1]\\
        \vdots & & \vdots\\
        \gamma_{R1}[n-N+1] & \cdots &\gamma_{RR}[n-N+1]\\
    \end{bmatrix} z^n := \sum_{n=0}^{2(N-1)} \boldsymbol{\Gamma}[n]z^n\label{eq:autocorrMatrixPoly}
\end{equation}
where $\lbrace \boldsymbol{\Gamma}[n] \in \bbC^{R\times R}\rbrace_{n=-N+1}^{N-1}$ is the auto-correlation matrix sequence of the $D$-dimensional vector signal $\lbrace \bfx[n] \in \bbC^R\rbrace_{n=0}^{N-1}$.
Plugging \eqref{eq:univariateFactor} into \eqref{eq:autocorrMatrixPoly} yields the rank-one autocorrelation matrix factorization problem \eqref{prob:PAF}.

\end{appendices}
\bibliography{refs}
\end{document}

%% file: notations.tex
\usepackage{lipsum}

\usepackage{amsfonts}
\usepackage{graphicx}
\usepackage{epstopdf}
\usepackage{algorithmic}
\usepackage[algo2e, algoruled]{algorithm2e}
\usepackage{tikz}
\usepackage{amsmath}
\usepackage{amssymb}
\usepackage{amsthm}
\usepackage{xcolor}
\usepackage[colorlinks, linkcolor=MidnightBlue, citecolor=green!50!black]{hyperref}
\usepackage[capitalize]{cleveref}

\usetikzlibrary{matrix,arrows.meta}

\ifpdf
  \DeclareGraphicsExtensions{.eps,.pdf,.png,.jpg}
\else
  \DeclareGraphicsExtensions{.eps}
\fi

\usepackage{enumitem}
\setlist[enumerate]{leftmargin=.5in}
\setlist[itemize]{leftmargin=.5in}

\crefname{hypothesis}{Hypothesis}{Hypotheses}

\usepackage{amsopn}

\usetikzlibrary{shapes,arrows}

\ifpdf
  \DeclareGraphicsExtensions{.eps,.pdf,.png,.jpg}
\else
  \DeclareGraphicsExtensions{.eps}
\fi

\usepackage{enumitem}
\setlist[enumerate]{leftmargin=.5in}
\setlist[itemize]{leftmargin=.5in}

\foreach \x in {a,...,z}{
    \expandafter\xdef\csname bf\x \endcsname{\noexpand\ensuremath{\noexpand\mathbf{\x}}}
    \expandafter\xdef\csname bm\x \endcsname{\noexpand\ensuremath{\noexpand\boldsymbol{\x}}}
  }
\foreach \x in {A,...,Z}{
    \expandafter\xdef\csname bs\x \endcsname{\noexpand\ensuremath{\noexpand\boldsymbol{\x}}}
    \expandafter\xdef\csname bf\x \endcsname{\noexpand\ensuremath{\noexpand\mathbf{\x}}}
    \expandafter\xdef\csname bb\x \endcsname{\noexpand\ensuremath{\noexpand\mathbb{\x}}}
    \expandafter\xdef\csname ds\x \endcsname{\noexpand\ensuremath{\noexpand\mathds{\x}}}
    \expandafter\xdef\csname cal\x \endcsname{\noexpand\ensuremath{\noexpand\mathcal{\x}}}
  }

\def\transp{\top}

\newcommand{\conj}[1]{\overline{#1}}
\renewcommand{\bmj}{\jmath}

\usepackage{amsthm}
\newtheorem{definition}[]{Definition}
\newtheorem{theorem}[definition]{Theorem}
\newtheorem{remark}[definition]{Remark}
\newtheorem{corollary}[definition]{Corollary}
\newtheorem{proposition}[definition]{Proposition}
\newtheorem{example}[definition]{Example}
\newtheorem{lemma}[definition]{Lemma}

\def\multmat#1#2{\mathbf{M}_{#2}(\mathbf{#1})}
\def\pco#1#2{#1[{#2}]}

\newcommand{\ie}{i.e.,~}

\usepackage{todonotes}